\documentclass[11pt,reqno]{amsart}

\usepackage{amssymb,amsmath,epsfig}
\usepackage{amsfonts}
\usepackage{pst-plot}
\usepackage{verbatim}
\newcommand{\Z}{\mathbb{Z}}

\newcommand{\C}{\mathbb{C}}
\newcommand{\N}{\mathbb{N}}
\newcommand{\NF}{NF}

\newtheorem{theorem}{Theorem}[section]

\newtheorem{lemma}[theorem]{Lemma}
\newtheorem{corollary}[theorem]{Corollary}
\newtheorem{proposition}[theorem]{Proposition}
\newtheorem*{theorem*}{Theorem}

\theoremstyle{remark}
\newtheorem*{remark}{Remark}
\newtheorem*{remarks}{Some remarks}

\numberwithin{equation}{section}

\newcommand{\evar}{d}

\date{\today} 

\begin{document} 
\title[Inequalities for full rank differences]{Inequalities for full rank differences of 2-marked Durfee symbols}
\author{Kathrin Bringmann}
\address{Mathematical Institute\\University of
Cologne\\ Weyertal 86-90 \\ 50931 Cologne \\Germany}
\email{kbringma@math.uni-koeln.de}
\email{bkane@math.uni-koeln.de}
\author{Ben Kane}
\thanks {The first author was partially supported by NSF grant DMS-0757907 and the Alfried-Krupp prize.}
\begin{abstract}
In this paper, we obtain infinitely many non-trivial identities and inequalities between full rank differences for $2$-marked Durfee symbols, a generalization of partitions introduced by Andrews.  A certain strict inequality, which almost always holds, shows that identities for Dyson's rank, similar to those proven by Atkin and Swinnerton-Dyer, are quite rare.  By showing an analogous strict inequality, we show that such non-trivial identities are also rare for the full rank, but on the other hand we obtain an infinite family of non-trivial identities, in contrast with the partition theoretic case.
\end{abstract}
\maketitle

\section{Introduction}

A {\it partition} of a non-negative integer $n$ is any non-increasing
sequence of positive integers whose sum is  $n$. As usual, let
$p\left(n\right)$ denote the number of partitions of $n$.  
The partition function satisfies the famous   ``Ramanujan congruences''
declaring that for all $n\geq 0$
\begin{align*}
p\left(5n+4\right)&\equiv 0\pmod{5},\\
p\left(7n+5\right)&\equiv 0\pmod{7},\\
p\left(11n+6\right)&\equiv 0\pmod{11}.
\end{align*}
In order to understand the congruences modulo $5$ and $7$ from a combinatorial point of view,  Dyson defined the  {\it rank of a partition} as  its  largest part minus its number of parts \cite{Dy}.   To simplify notation, for integers $0\leq r<t$,  we let $N\left(r,t;n\right)$ be the number of partitions of $n$ whose rank is congruent to $r$ modulo $t$ and we will denote the corresponding generating function by 
\begin{equation}\label{eqn:gtdef}
g_t\left(r\right)=g_t\left(r;q\right):=\sum_{n=0}^{\infty} N\left(r,t;n\right) q^n.
\end{equation}
Dyson conjectured that the congruence for $5n+4$ is explained by the fact that the rank modulo $5$ divides the partitions of $5n+4$ into $5$ equally sized classes, namely for every $r,s\in \Z$
\begin{equation}\label{eqn:mod5}
N\left(r,5;5n+4\right)=N\left(s,5,5n+4\right)
\end{equation}
holds for all $n\in \N_0$.  This implies the above congruence, since by \eqref{eqn:mod5}
$$
p\left(5n+4\right)=5N\left(0,5,5n+4\right)\equiv 0\pmod{5}.
$$
Similarly, Dyson conjectured that the congruence modulo $7$ is explained by the identity 
\begin{equation}\label{eqn:mod7}
N\left(r,7;7n+5\right)=N\left(s,7;7n+5\right)
\end{equation}
for all $r,s\in \Z$ and $n\in \N_0$.  
Dyson's rank conjectures were later proved by Atkin and Swinnerton-Dyer \cite{AtkinSwinnertonDyer1}.  On the other hand, Dyson's rank fails to divide the partitions of $11n+6$ in the same way, and he famously conjectured the existence of a new statistic which he called the ``crank'' and which would explain all three congruences simultaneously.  This statistic remained hidden until a proper definition was finally found in work of Andrews and Garvan \cite{AndrewsGarvan1,Garvan1}. 

Many further congruences exist for the partition function.  Their proofs frequently go through automorphic properties of certain generating functions.  For example, the first author and Ono \cite{BringmannOno2} have realized the rank generating function as the holomorphic part of a \begin{it}harmonic weak Maass form\end{it}, a certain non-holomorphic modular form (cf. \cite{BruinierFunke1} for a definition).  The special case 
$$
f\left(q\right):=1+\sum_{n=1}^{\infty} \Big(N\left(0,2;n\right)-N\left(1,2;n\right)\Big)q^n = 1+\sum_{n=1}^{\infty} \frac{q^{n^2}}{\left(-q\right)_n^2},
$$
with $\left(a;q\right)_n=\left(a\right)_n:=\left(1-a\right)\left(1-aq\right)\cdots \left(1-aq^{n-1}\right)$, is one of Ramanujan's mock theta functions and its place in the theory of automorphic forms was first realized by Zwegers \cite{Zwegers2}.  Using the theory of harmonic weak Maass forms, the first author and Ono \cite{BringmannOno2} have shown that for every prime power $p^j$ relatively prime to $6t$ (an extension to include the case when $t=p^\ell$ was given by the first author \cite{Bringmann2}) there are infinitely many non-nested arithmetic progressions $An+B$ ($0\leq B<A$) for which every $0\leq r<t$ and $n\in \N_0$ satisfy the congruence
\begin{equation}\label{eqn:BOcong}
N\left(r,t;An+B\right)\equiv 0\pmod{p^j}.
\end{equation}
For these choices of $A$ and $B$, this gives a refinement of the congruence
$$
p\left(An+B\right)\equiv 0 \pmod{p^j},
$$
which for $j=1$ were previously proven by Ono \cite{Ono1} and for $j>1$ were proven by Ahlgren and Ono \cite{AhlgrenOno1}.

The abundance of such congruences lead one naturally to ask which of these follow by equalities of the type given in \eqref{eqn:mod5} and \eqref{eqn:mod7} and to investigate in general when such equalities exist.  By considering conjugate partitions, one easily sees for every $1\leq r<t$ and every $n\in \N$ that 
\begin{equation}\label{eqn:conjugateidentity}
N\left(r,t;n\right)=N\left(t-r,t;n\right).
\end{equation}
Due to these trivial identities we may assume throughout the paper that $0\leq r<s\leq \frac{t}{2}$.  Under this assumption, we see that apart from the trivial identities in \eqref{eqn:conjugateidentity}, other identities such as \eqref{eqn:mod5} and \eqref{eqn:mod7} turn out to be quite rare.  Indeed, based on asymptotic formulas for ranks shown by the first author \cite{Bringmann1} using the circle method, the authors have shown \cite{BringmannKane1} that for $t$ odd and $0\leq r<s<\frac{t}{2}$, there are only finitely many $\left(r,s,t\right)$ for which the identity
$$
N\left(r,t;n\right)=N\left(s,t;n\right)
$$
holds for infinitely many $n\in \N$.  Specifically, there are infinitely many such $n$ if and only if $t\in \left\{5,7\right\}$ or $t=9$ and $\left(r,s\right)\in \left\{ \left(0,4\right),\left(3,4\right) \right\}$.

\begin{theorem}[\cite{BringmannKane1}]\label{thm:RankDifferences}
Assume that $t\geq 11$ is an odd integer.
Then for $0\leq r<s\leq \frac{t-1}{2}$ we have for $n>N_{r,s,t}$, where $N_{r,s,t}$ is an explicit constant, the inequality
$$
N\left(r,t;n\right)> N\left(s,t;n\right). 
$$ 
\end{theorem} 
\begin{remark}
The theory of harmonic weak Maass forms has essentially reduced identities relating the ranks to a calculation of finitely many Fourier coefficients, but inequalities such as those contained in Theorem \ref{thm:RankDifferences} are more difficult to prove because they require a careful analysis of the asymptotic growth of the coefficients of these generating functions and cannot be proven by merely checking the inequality for finitely many Fourier coefficients.
\end{remark}
One sees quite clearly from Theorem \ref{thm:RankDifferences} why the rank fails to explain the congruence for $t=11$, whereas the behavior for $t<11$ is quite different.  One finds that in these cases the direction of the above inequality depends on the congruence class of $n$ modulo $t$.  Theorem \ref{thm:RankDifferences} essentially completed the determination of the congruence classes modulo $t$ exhibiting positivity, negativity, and equality.  The investigation into such inequalities for fixed small $t$ was initiated by Andrews and Lewis \cite{AndrewsLewis1} and Lewis \cite{Lewis1} (these theorems involve $t$ even), while the first author proved the inequalities for $t=3$ \cite{Bringmann1} which were conjectured in \cite{AndrewsLewis1}.  Although one expects a theorem similar to Theorem \ref{thm:RankDifferences} to hold for the crank modulo $t$ for all $t\geq t_0$ beyond some boundary $t_0$, it is clear that $t_0>11$ must hold, as the crank modulo $11$ divides the partitions of $11n+6$ into equally sized classes.  We note that the proof of the inequalities for the cranks differences would be easier, since while the rank generating function is the holomorphic part of a harmonic weak Maass form \cite{BringmannOno2}, the corresponding crank generating function is a holomorphic modular form.

A recent generalization of partitions called $k$-marked Durfee symbols, whose definition will be recalled in Section \ref{section:Durfee}, was given by Andrews \cite{Andrews1}.  He used these $k$-marked Durfee symbols to give a combinatorial interpretation of the $k$-th rank moments of partitions defined by Atkin and Garvan \cite{AtkinGarvan1}.  The $1$-marked Durfee symbols (or simply, Durfee symbols) are in one-to-one correspondence with partitions.  One is naturally led to a definition of a certain rank statistic for $k$-marked Durfee symbols, which Andrews called the \begin{it}($k$-th) full rank\end{it}.  By work of the first author, Garvan, and Mahlburg \cite{BGM1}, it follows that the generating function for those $k$-marked Durfee symbols with full rank congruent to $r$ modulo $t$ is a \begin{it}quasimock theta function\end{it}, which is essentially the holomorphic part of linear combinations of harmonic weak Maass forms and their derivatives (thus generalizing quasimodular forms). Indeed, they show that the analytic continuation of their (6.2) with $x_i=x^i$ may be written as a linear combination of terms of the type $\frac{\partial}{\partial x^r} R\left(x;q\right)$ and that specialization of $\frac{\partial}{\partial x^r} R\left(x;q\right)$ to a root of unity $x$ is a quasimock theta function.  Following the notation for the rank, we will denote the number of $k$-marked Durfee symbols of size $n$ with full rank congruent to $r$ modulo $t$ by $\NF_k\left(r,t;n\right)$.  A number of relations between the rank and the full rank leads one to search for identites such as those in equations \eqref{eqn:mod5} and \eqref{eqn:mod7} for the full rank.  For example, in the case $t=5$, Andrews proved in Theorem 17 of \cite{Andrews1} that for all integers $r,s\in \Z$ we have
\begin{equation}\label{eqn:fullmod5}
\NF_2\left(r,5;5n\pm 1\right)=\NF_2\left(s,5;5n\pm 1\right).
\end{equation}
Thus the full rank modulo $5$ divides the $k$-marked Durfee symbols of $5n\pm 1$ into 5 equally sized classes and gives a combinatorial explanation for the congruence
$$
\NF_2\left(5n\pm 1\right)\equiv 0\pmod{5},
$$
where $\NF_2\left(n\right)$ denotes the number of $2$-marked Durfee symbols of $n$.  Andrews also proved similar results for the modulus $t=7$.  Noting the automorphic properties for the corresponding generating functions as proven in \cite{BGM1}, \eqref{eqn:fullmod5} can again be reduced to a check of finitely many Fourier coefficients of the associated generating function.  In this paper, we will restrict to the case $k=2$ and employ the asymptotic growth of the coefficients of these generating functions to show that equalities such as \eqref{eqn:fullmod5} are again quite rare as we vary $t$.

As we will recall in Section \ref{section:Durfee}, the symmetry given by conjugation for the rank of partitions (or equivalently, $1$-marked Durfee symbols) generalizes to a symmetry for the full rank.  This gives trivial identities for the full rank such as those obtained in \eqref{eqn:conjugateidentity}.  For this reason, we may restrict ourselves to $0\leq r <s\leq \frac{t}{2}$.  Due to a technical difficulty occuring when $3\mid t$, we shall first assume that $\left(t,6\right)=1$.  In this case, one may prove a result for the full rank resembling Theorem \ref{thm:RankDifferences}.
\begin{theorem}\label{thm:GeneralFullRankInequality}
Suppose that $t>7$ is a positive integer with $\left(t,6\right)=1$.  Then for $0\leq r < s\leq \frac{t-1}{2}$ with $\left(r,s\right)\neq \left(1,2\right)$, we have, for sufficiently large $n$,
$$
\NF_2\left(r,t;n\right)> \NF_2\left(s,t;n\right).
$$
\end{theorem}
However, in contrast to the usual rank, in the case $\left(r,s\right)=\left(1,2\right)$ an infinite family (in the variable $t$) of non-trivial identities similar to \eqref{eqn:fullmod5} hold for all $t\neq 3$ odd and $n\in \N$.
\begin{theorem}\label{thm:FullRankIdentity}
\noindent
\begin{enumerate}
\item For every odd $t$ and for every $n\in \N$, we have 
$$
\NF_2\left(1,t;n\right)=\NF_2\left(2,t;n\right).
$$
\item
For $t$ even and $n\in \N$, we have
$$
\NF_2\left(1,t;n\right)\leq \NF_2\left(2,t;n\right),
$$
and equality holds if and only if $n\in \left\{ 0,\dots, \frac{t}{2}, \frac{t}{2}+2\right\}$.
\item For every $n\in \N$ one has the equality
$$
\NF_2\left(1;n\right)=\NF_2\left(2;n\right),
$$
where $\NF_2\left(r;n\right)$ denotes the number of 2-marked Durfee symbols with full rank equal to $r$.
\end{enumerate}
\end{theorem}
\begin{remarks}
\noindent
\begin{enumerate}
\item In the case that $t\neq 3$ is odd, Theorem \ref{thm:FullRankIdentity} (1) shows that there is always an identity which does not come from the aforementioned trivial conjugation symmetry, in contrast with the corresponding result for partitions.
\item 
Although Theorem \ref{thm:RankDifferences} only applies to the case when $t$ is odd, a similar result is expected for the rank when $t$ is even.  Hence in this case it is interesting to note that Theorem \ref{thm:FullRankIdentity} (2) gives an inequality for the full rank in the opposite direction of what is expected for the rank.
\item The inequalities implied by Theorem \ref{thm:FullRankIdentity} (2) are each proven through an identity followed by an injective map from one type of partitions into another.  Hence the full content of Theorem \ref{thm:FullRankIdentity} is really concerned with identities.  Such identities may theoretically be proven by using the theory of harmonic weak Maass forms to show that both sides of the identity correspond to the same harmonic weak Maass form.  However, since the calculation would get quite tedious, we choose a more direct approach in this paper.
\end{enumerate}
\end{remarks}

In addition to showing identities such as those in equations \eqref{eqn:mod5} and \eqref{eqn:mod7}, Atkin and Swinnerton-Dyer also proved that in many cases the difference of two rank modulo $5$ and $7$ generating functions are modular forms.  For example, they showed that 
 \begin{equation}\label{eqn:hickersonexample}
\sum_{n=0}^{\infty}\Big(N\left(0,7;7n+6\right)- N\left(1,7;7n+6\right)\Big)q^n =  -\frac{\left(q;q^7\right)_{\infty}^2\left(q^6;q^7\right)_{\infty}^2\left(q^7;q^7\right)_{\infty}^2}{\left(q\right)_{\infty}}.
\end{equation} 
Such identities are far more abundant, and are now explained by the fact that the rank generating function is the holomorphic part of a harmonic weak Maass form and hence certain differences are modular forms, with a large infinite class of such modular differences proven by the first author, Ono, and Rhoades \cite{BringmannOnoRhoades1}.   This has led to a further investigation of the relevant harmonic weak Maass forms in order to establish identities such as \eqref{eqn:hickersonexample} (for example, \cite{AhlgrenTreneer1,Dewar1}).

Based on a relationship between the full rank generating functions and those of the rank generating functions, we show in Section \ref{section:ASD-identities} a number of infinite product identities for the full rank paralleling equation \eqref{eqn:hickersonexample}.  For example, we obtain the equality
\begin{equation}\label{eqn:infprod}
\sum_{n=0}^{\infty}\Big(\NF_2\left(0,7;7n+3\right)-\NF_2\left(1,7;7n+3\right)\Big)q^n = \frac{\left(q^7;q^7\right)_{\infty}}{\left(q^2;q^7\right)_{\infty}\left(q^5;q^7\right)_{\infty}}.
\end{equation}
We note that in the case $t=5$, the identities in Section \ref{section:ASD-identities} were proven by Keith in Theorems 1 and 2 of \cite{Keith}.  He considers general $k$, but restricts himself to the special case $t=2k+1$ and exploits identities of the type
$$
\NF_k\left(r,2k+1;n\right)=\NF_k\left(s,2k+1;n\right)
$$
when $\left(r,2k+1\right)=\left(s,2k+1\right)$.  Theorem \ref{thm:GeneralFullRankInequality} implies that such identities are rare when we restrict to the case $k=2$ but allow general modulus $t$, so we include Keith's result in Section \ref{section:ASD-identities} in order to list all tuples $\left(r,s,t\right)$ which give equalities of this type.

The paper is organized as follows.  We give the definition of $k$-marked Durfee symbols and the full rank in Section \ref{section:Durfee}.  In Section \ref{section:rankfullrank} we show a linear relationship between the generating function
\begin{equation}\label{eqn:adjsumtype}
f_t\left(r,r+1\right)=f_t\left(r,r+1;q\right):=\sum_{n=0}^{\infty}\Big(\NF_2\left(r,t;n\right)-\NF_2\left(r+1,t;n\right)\Big)q^n
\end{equation}
and rank generating functions, which will be the basis for most of our results.  In Section \ref{section:FullRankIdentity}, we show Theorem \ref{thm:FullRankIdentity}, which relates $\NF_2\left(1,t;n\right)$ to $\NF_2\left(2,t;n\right)$.  We show infinite product identities for the full rank such as \eqref{eqn:infprod} in Section \ref{section:ASD-identities}.  Building on the work from Section \ref{section:rankfullrank}, Section \ref{section:FullRankInequality} is devoted to proving the inequality given in Theorem \ref{thm:GeneralFullRankInequality} by showing that all but finitely many coefficients of $f_t\left(r,r+1\right)$ are positive under proper restrictions for $r$.  In Section \ref{section:smallinequalities}, we conclude the paper with a series of inequalities between $\NF_2\left(r,t;n\right)$ and $\NF_2\left(s,t;n\right)$ for small $t$.

\section*{Acknowledgements}
The authors thank W. Keith for making available his preprint prior to publication.  The authors would also like to thank the referees for helpful and detailed reports.

\section{Durfee symbols}\label{section:Durfee}
In this section we will recall Andrews' definitions \cite{Andrews1} for Durfee symbols, $k$-marked Durfee symbols, and the full rank for $k$-marked Durfee symbols.  The \begin{it}Durfee symbols of size $n$\end{it} are given by an integer $d$ and two nonincreasing sequences of integers $d\geq a_1\geq \cdots \geq a_m$ and $d\geq b_1\geq \cdots\geq b_{\ell}$ as in the following representation:
$$
\left(\begin{matrix} a_1&\dots&a_m\\ b_1&\dots & b_{\ell}\end{matrix}\right)_d,
$$
so that 
$$
d^2+\sum_{i=1}^{m} a_i + \sum_{j=1}^{\ell}b_j=n.
$$
Recall that the largest square in the Ferrers diagram 
of a partition is referred to as the \begin{it}Durfee square\end{it} of the partition.  Then the above Durfee symbol corresponds to the partition with Durfee square of size $d$, columns of length $a_1,\dots,a_m$ to the right of the Durfee square, and rows of length $b_1,\dots, b_{\ell}$ below the Durfee square.  For example, the Durfee symbol
$$
\left(\begin{matrix} 3&1&1\\ 2&1 & \end{matrix}\right)_3
$$
corresponds to the partition $\left(6,4,4,2,1\right)$.  Notice that the rank of the partition corresponding to a Durfee symbol is precisely $m-\ell$, the length of the first row of the Durfee symbol minus the length the second row.  In the above example, this gives a partition of rank $1$.

To define $k$-marked Durfee symbols, we require $k$ copies of the integers, so that each element $a_i$ and $b_j$ is given a subscript between $1$ and $k$ indicating which copy of the integers it is contained in.  A \begin{it}$k$-marked Durfee symbol\end{it} is a Durfee symbol with the following restrictions on the allowable parts $a_i$, $b_j$ and their subscripts:
\noindent
\begin{enumerate}
\item  The sequence of parts and the sequence of subscripts in each row must be non-increasing.
\item  Each of the subscripts $1,2,\dots,k-1$ occurs at least once in the top row.
\item  If $M_1, N_2, \dots V_{k-2},W_{k-1}$ are the largest parts with their respective subscripts in the top row, then all parts in the bottom row with subscript 1 lie in $[1,M]$, with subscript $2$ lie in $[M,N]$, $\dots$, with subscript $k-1$ lie in $[V,W]$, and with subscript $k$ lie in $[W,S]$, where $S$ is the side of the Durfee square.
\end{enumerate}
For a $k$-marked Durfee symbol $\delta$, let $\tau_i\left(\delta\right)$ (resp. $\beta_i\left(\delta\right)$) be the number of entries in the top (resp. bottom) row with subscript $i$.  Then the \begin{it}$i$-th rank\end{it} is defined by 
$$
\rho_i\left(\delta\right):=
\begin{cases} 
\tau_i\left(\delta\right) - \beta_i\left(\delta\right)-1 & \text{for }1\leq i <k,\\
\tau_i\left(\delta\right) - \beta_i\left(\delta\right) & \text{for }i=k.
\end{cases}
$$
We refer to $\sum_{i=1}^{k} i \rho_i\left(\delta\right)$ as the \begin{it}($k$-th) full rank\end{it} of $\delta$.  

The symmetry of conjugation for the rank is observed in the Durfee symbol as simply swapping the first and second rows.  For the $k$-th full rank, one defines a similar conjugation action.  For $\delta$ a $k$-marked Durfee symbol, we define a conjugate $k$-marked Durfee symbol $\overline{\delta}$ as follows.  One first interchanges the parts with subscript $k$ in the top and bottom rows as in the $k=1$ case.  For $1\leq i<k$, the largest part with subscript $i$ remains in the top row, while all other parts with subscript $i$ from the top and bottom rows are interchanged.  This preserves condition (3) in the definition and the $i$-th rank of $\overline{\delta}$ is
$$
\rho_i\left(\overline{\delta}\right) = \left(1+\beta_i\left(\delta\right)\right) - \left(\tau_i\left(\delta\right)-1\right) - 1 = - \rho_i\left(\delta\right).
$$
Hence one may always obtain trivial identities such as those in \eqref{eqn:conjugateidentity} by the symmetry
\begin{equation}\label{eqn:frankconj}
\sum_{i=1}^{k}i\rho_i\left(\overline{\delta}\right) = - \sum_{i=1}^{k}i\rho_i\left(\delta\right).
\end{equation}
We will be interested in observing other equalities which do not follow from the above observation.  In order to do so, we will restrict ourselves to the case $k=2$ and work with a relation given between the generating functions for the usual rank for partitions and the full rank on $2$-marked Durfee symbols.

\section{Relating the full rank to the classical rank}\label{section:rankfullrank}
Much as Atkin and Swinnerton-Dyer did for partitions, we will work with the generating function for full ranks of $2$-marked Durfee symbols.  We begin with a series of necessary definitions.  

By classifying partitions in terms of the size of the Durfee square, the classical rank generating function is given (cf. \cite{HardyWright1} Chapter 18, Section 19.7, Lemma 7.9 of \cite{Garvan1}) by
\begin{equation}\label{eqn:Rrank}
R\left(w;q\right):=
1+\sum_{n=1}^{\infty}\sum_{m=-\infty}^{\infty} N\left(m,n\right)w^mq^n= \sum_{n=0}^{\infty} \frac{q^{n^2}}{\left(wq\right)_n \left(w^{-1}q\right)_n}= 
\frac{\left(1-w \right)}{\left(q\right)_{\infty}} 
\sum_{n \in \Z} 
\frac{\left(-1\right)^n \, q^{\frac{n}{2}\left(3n+1\right) }}{1 - w q^n}.
\end{equation}
We define the generating function
\begin{equation}\label{eqn:gtrsdef}
g_t\left(r,s\right)=g_t\left(r,s;q\right):=
 \sum_{n=0}^{\infty} 
 \Big( N\left(r,t;n\right)  -N\left(s,t;n\right) \Big)q^n 
 = 
 \frac{1}{t}
  \sum_{j=1}^{t-1}
  \left(  \zeta_t^{rj}  -\zeta_t^{sj} \right)
 R\left(\zeta_t^j;q\right).
\end{equation}
Following Andrews \cite{Andrews1}, we consider
\begin{align}\label{eqn:R2def}R_2\left(x_1,x_2;q\right)&:=\sum_{\substack{m_1>0\\ m_2\geq 0}} \frac{q^{\left(m_1+m_2\right)^2+m_1}}{\left(x_1q\right)_{m_1}\left(q/x_1\right)_{m_1}\left(x_2q^{m_1}\right)_{m_2+1}\left(q^{m_1}/x_2\right)_{m_2+1}}\\
\nonumber & = 
\frac{1}{\left(q\right)_{\infty}}\sum_{n=1}^{\infty} \frac{\left(-1\right)^{n-1}\left(1+q^n\right)\left(1-q^n\right)^2 q^{\frac{n}{2}\left(3n+1\right)}}{\left(1-x_1 q^n\right)\left(1-q^n/x_1\right)\left(1-x_2 q^n\right)\left(1-q^n/x_2\right)},
\end{align}
where the equality comes from Theorem 3 in \cite{Andrews1}.  By Theorem 10 of \cite{Andrews1}, $R\left(x_1,x_2;q\right)$ is the generating function for $2$-marked Durfee symbols with the exponent of $x_i$ counting the $i$-th rank.  
Hence $R\left(x,x^2;q\right)$ is the generating function for $2$-marked Durfee symbols with the exponent of $x$ counting the full rank.  In Corollary 8 of \cite{Andrews1}, Andrews concludes from \eqref{eqn:R2def} the relation 
\begin{equation}\label{eqn:R2expand}
R_2\left(x,x^2;q\right) = \frac{R\left(x;q\right)-R\left(x^2;q\right)}{\left(x-x^2\right)\left(1-x^{-3}\right)}
\end{equation}
whenever $x$ is not a third root of unity or zero.

Analogous to \eqref{eqn:gtrsdef}, we may now define the generating function
 \begin{equation}\label{eqn:ftrsdef}
 f_t\left(r,s\right)=f_t\left(r,s;q\right):=
 \sum_{n=0}^{\infty} 
 \Big( \NF_2 \left(r,t;n\right)  -\NF_2\left(s,t;n\right) \Big)q^n 
 = 
 \frac{1}{t}
  \sum_{j=1}^{t-1}
  \left(  \zeta_t^{rj}  -\zeta_t^{sj} \right)
 R_2 \left(\zeta_t^j,\zeta_t^{2j};q \right).
 \end{equation}
If $3\nmid t$, then, using \eqref{eqn:R2expand}, this simplifies as  
 \begin{equation}\label{eqn:ftrsdefmod3}
 f_t\left(r,s\right)=  
 -\frac{1}{t} \sum_{j=1}^{t-1}
 \frac{\zeta_t^{-rj} -\zeta_t^{-sj}}{\left( \zeta_t^{2j} - \zeta_t^{j} \right) \left( 1- \zeta_t^{-3j}\right)} 
 \left( 
 R \left(\zeta_t^j;q \right) - R \left( \zeta_t^{2j};q\right)
 \right).
 \end{equation}

We will make constant usage of the symmetries coming from conjugation, given by
\begin{equation}\label{eqn:conjugation}
g_t\left(-r,s\right)=g_t\left(-r,-s\right)=g_t\left(r,-s\right)=g_t\left(r,s\right)=-g_t\left(s,r\right)
\end{equation}
and, coming from  \eqref{eqn:frankconj} (or $x_i\to x_i^{-1}$ in \eqref{eqn:R2def}),
\begin{equation}\label{eqn:fullrankconjugation}
f_t\left(-r,-s\right)=f_t\left(-r,s\right)=f_t\left(r,s\right)=-f_t\left(s,r\right).
\end{equation}

Define further the generating function for difference of ranks in congruence classes by
\begin{equation}\label{eqn:gtrscongdef}
g_{t,d}\left(r,s\right)=g_{t,d}\left(r,s;q\right):=\sum_{n=0}^{\infty} \Big( N\left(r,t;tn+d\right) - N\left(s,t;tn+d\right) \Big) q^n,
\end{equation}
and the generating function for corresponding difference of full ranks accordingly by
\begin{equation}\label{eqn:ftrscongdef}
f_{t,d}\left(r,s\right)=f_{t,d}\left(r,s;q\right):=\sum_{n=0}^{\infty} \Big( \NF_2\left(r,t;tn+d\right) - \NF_2\left(s,t;tn+d\right) \Big) q^n.
\end{equation}

The purpose of this section will be to establish the following identity relating the difference of full ranks for adjacent congruence classes to differences of ranks.
\begin{proposition}\label{prop:GeneralSetupIdentity}
When $t$ is odd, we obtain the following equality
\begin{equation}\label{eqn:adjfinal}
f_t(r,r+1)=\frac{1}{t} \sum_{m=0}^{t-1} \left({t-1}-m\right) g_t\Big(-3m+r-1,\overline{2}(-3m+r-1)\Big) +\frac{3}{t} \delta_{3\mid t} f_3(r,r+1),
\end{equation}
where $\overline{2}=\frac{t+1}{2}$ denotes the multiplicative inverse of $2$ modulo $t$ and $\delta_{3\mid t}=1$ if $3$ divides $t$ and $0$ otherwise.
\end{proposition}
\begin{proof}
Fix $t$ odd and a primitive $t$-th root of unity $\zeta_t$.  We will use the fact that for any $t$-th root of unity $\zeta$ one has
\begin{equation}\label{eqn:zetainv}
\frac{1}{1-\zeta}=\frac{1}{t}\sum_{m=0}^{t-1} \Big(t-1-m\Big) \zeta^m.
\end{equation}

Using \eqref{eqn:R2expand} together with \eqref{eqn:ftrsdefmod3}, \eqref{eqn:fullrankconjugation}, and \eqref{eqn:zetainv}, we now expand 
\begin{multline}
\label{eqn:adjexpand}
f_t(r,r+1)-\frac{3}{t}\delta_{3\mid t} f_3(r,r+1)=f_t(-r,-(r+1))-\frac{\delta_{3\mid t}}{t}\sum_{j=\frac{t}{3},\frac{2t}{3}} \left(\zeta_t^{rj} - \zeta_t^{(r+1)j}\right)R_2\left(\zeta_t^j,\zeta_t^{2j};q\right)\\
=\frac{1}{t^2}\sum_{\substack{1\leq j\leq t-1\\j\neq \frac{t}{3},\frac{2t}{3}}} \sum_{m=0}^{t-1} \Big(t-1-m\Big) \zeta_t^{\big(-3m+r-1\big)j}  \left(R\left(\zeta_t^{j};q\right)-R\left(\zeta_t^{2j};q\right)\right).
\end{multline}
However, for every $3\mid t$ and $m\in \Z$ one has that
\begin{equation}\label{eqn:3rdroot}
\sum_{j=\frac{t}{3},\frac{2t}{3}} \zeta_t^{\left(-3m + r-1\right)j}\left(R\left(\zeta_t^j;q\right)-R\left(\zeta_t^{2j};q\right)\right) =\sum_{j=1}^2 \zeta_3^{(r-1)j}\left(R\left(\zeta_3^j;q\right) - R\left(\zeta_3^{2j};q\right)\right)=0.
\end{equation}
Hence when $3\mid t$ we may add \eqref{eqn:3rdroot} to \eqref{eqn:adjexpand} without changing the sum.  We then reverse the order of summation and split the sum (completed to $1\leq j\leq t-1$) in \eqref{eqn:adjexpand} into two sums coming from $R\left(\zeta_t^{j};q\right)$ and $R\left(\zeta_t^{2j};q\right)$.  
Since the sum on $j$ only depends on $j$ modulo $t$, we may then make the change of variables $j\to \overline{2}j$ in the second sum to see by \eqref{eqn:gtrsdef} that \eqref{eqn:adjexpand} equals
\begin{equation}\label{eqn:adjrelprime3}
\frac{1}{t} \sum_{m=0}^{t-1} \Big(t-1-m\Big) g_t\left(-3m+r-1,\overline{2}\big(-3m+r-1\big)\right),
\end{equation}
completing the proof.
\end{proof}

When $r\equiv 1\pmod{3}$, we can use Proposition \ref{prop:GeneralSetupIdentity} to prove the following rather pleasant identity.
\begin{proposition}\label{prop:GeneralOneMod3Identity}
For every odd positive integer $t$ and every integer $1\leq r\leq 3t+1$ with $r\equiv 1\pmod{3}$, one has 
$$
f_t(r,r+1) = \sum_{m=1}^{\frac{r-1}{3}} g_t\left(3m,\overline{2}\cdot 3m\right).
$$
\end{proposition}
\begin{remark}
We remark that the equality in Proposition \ref{prop:GeneralOneMod3Identity} holds true for any $r\in \N$ with $r \equiv 1\pmod{3}$, but we only prove the cases $1\leq r \leq 3t+1$ because these are sufficient for the purposes of this paper.
\end{remark}
\begin{proof}
We will denote $j:=\frac{r-1}{3}$.  Since $f_3(r,r+1)=f_3(1,2)=0$, making the change of variables $m\to m+j$ in Proposition \ref{prop:GeneralSetupIdentity} yields 
$$
f_t(r,r+1)=\frac{1}{t}\sum_{m=-j}^{t-1-j} \Big(t-1-(m+j)\Big) g_t\left(-3m,-\overline{2}\cdot 3m\right).
$$
Notice that if we let $\widetilde{m}:=m+t$, then, since $g_t(a,b)$ only depends on $a$ and $b$ modulo $t$,
$$
\Big(t-1-(m+j)\Big) g_t\left(-3m,-\overline{2}\cdot 3m\right)  =  t \cdot g_t\left(-3\widetilde{m},-\overline{2}\cdot 3\widetilde{m}\right) + \Big(t-1-\left(\widetilde{m}+j\right)\Big) \cdot g_t\left(-3\widetilde{m},-\overline{2}\cdot3\widetilde{m}\right).
$$
Using this fact, we make the change of variables $m\to t-m$ for $-j\leq m\leq -1$ to obtain
\begin{equation}\label{eqn:shiftbottom}
f_t\left(r,r+1\right)=\frac{1}{t} \sum_{m=0}^{t-1} \Big(t-1-(m+j)\Big) g_t\left(-3m,-\overline{2}\cdot 3m\right) + \sum_{m=t-j}^{t-1} g_t\left(-3m,-\overline{2}\cdot 3m\right).
\end{equation}
We then write the first sum twice and for $m\neq 0$ we group the $m$ and $t-m$ terms together.  Using the symmetry \eqref{eqn:conjugation} with the change of variables $m\to t-m$ and $g_t\left(0,0\right)=0$, the first sum becomes
$$
\frac{t-2-2j}{2t}\sum_{m\pmod{t}}g_t\left(3m,\overline{2}\cdot 3m\right)=0.
$$
The fact that this is zero follows by splitting into two sums using $g_t(r,s)=g_t(r)-g_t(s)$ and then making the change of variables $m\to 2m$ in the second sum.  The result of the proposition then follows by making the change of variables $m\to t-m$ in the second sum of \eqref{eqn:shiftbottom}.
\end{proof}
One immediately obtains the following simple identities as a corollary.
\begin{corollary}\label{cor:General47Identity}
For every odd positive integer $t$ one has 
$$
f_t(4,5) = g_t\left(3,\frac{t-3}{2}\right) \qquad\text{ and }\qquad f_t(7,8)= g_t\left(6,\frac{t-3}{2}\right).
$$
\end{corollary}

For later usage in the proof of Theorem \ref{thm:GeneralFullRankInequality}, we now rewrite Proposition \ref{prop:GeneralOneMod3Identity} in a form which will prove beneficial for showing inequalities.
\begin{lemma}\label{lem:1mod3}
Suppose that $t$ is odd and $1< r\leq 3t+1$ satisfies $r\equiv 1\pmod{3}$.  Then we have 
\begin{equation}\label{eqn:rmod1arbitrary}
f_t(r,r+1)= \sum_{m=0}^{\left\lceil \frac{r-1}{6}\right\rceil -1} g_t\left(r-1 -3m,\frac{ t-3}{2} - 3m\right).
\end{equation}
\end{lemma}
\begin{proof}
We write $r=1+3j$.  By Proposition \ref{prop:GeneralOneMod3Identity}, we have
\begin{equation}\label{eqn:tosplit}
f_{t}(r,r+1) = \sum_{m=1}^{j} g_t\left(3m,\overline{2}\cdot 3m\right).
\end{equation}
We now split $g_t(a,b)=g_t(a)-g_t(b)$ and break \eqref{eqn:tosplit} into two sums.  The terms with $m$ even from the second sum then cancel the first $\left\lfloor\frac{j}{2}\right\rfloor$ terms from the first sum.  Hence
\begin{equation}\label{eqn:oddsleft}
f_t(r,r+1) = \sum_{m=\left\lfloor \frac{j}{2}\right\rfloor +1 }^{j} g_t\left(3m\right) - \sum_{\substack{1\leq m\leq j\\ m\text{ odd}}} g_t\left(\overline{2}\cdot 3m\right).
\end{equation}
Writing $m=2\ell+1$ in the second sum and using the symmetry \eqref{eqn:conjugation} while making the shift $m\to j-m$ in the first sum of \eqref{eqn:oddsleft} and recombining with $g_t(a,b)=g_t(a)-g_t(b)$ yields \eqref{eqn:rmod1arbitrary}.
\end{proof}

\section{Relations between $\NF_2(1,t;n)$ and $\NF_2(2,t;n)$}\label{section:FullRankIdentity}
In this section we will prove Theorem \ref{thm:FullRankIdentity} for the difference of full rank generating functions $f_t(1,2)$.  For $t$ odd, Proposition \ref{prop:GeneralOneMod3Identity} immediately implies Theorem \ref{thm:FullRankIdentity} (1), since the sum in the proposition is empty.   Theorem \ref{thm:FullRankIdentity} (3) now follows immediately from Theorem \ref{thm:FullRankIdentity} (1) by taking $t$ odd with $t\to \infty$ for each $n$ fixed.

In order to prove Theorem \ref{thm:FullRankIdentity} (2), we will show a relationship between $f_t(1,2)$ and $g_t\left(\frac{t}{2}\right)$, defined in \eqref{eqn:gtdef}, in the case when $t$ is even.
\begin{proposition} \label{prop:FullRankEvenEquality}
For every even positive integer $t$, we have
$$
f_t(1,2) = -\frac{1}{2}g_t\left(\frac{t}{2}\right).
$$
\end{proposition}
\begin{proof}
Define
\begin{equation}\label{eqn:gdef}
G\left(x;q\right):= \frac{1}{(q)_{\infty}}\sum_{n=1}^{\infty} \frac{(-1)^{n-1}\left(1+q^n\right)\left(1-q^n\right)^2q^{\frac{n}{2}(3n-1)}} {\left(1-x q^{n}\right)\left(1-q^{n}/x\right)} .
\end{equation}
Using Euler's pentagonal number theorem (for example, see \cite{AndrewsPartitions}, Chapter 1), namely
$$
\left(q\right)_{\infty}= 1+ \sum_{n=1}^{\infty} \left(-1\right)^n q^{ \frac{n}{2}\left(3n-1\right)}\left( 1+q^n\right),
$$
along with the last equality of \eqref{eqn:Rrank}, for every $x\in \C$ we obtain 
\begin{equation}\label{eqn:geqn}
R\left(x;q\right)=1+G\left(x;q\right).
\end{equation}
Decomposing the summand of \eqref{eqn:R2def} into partial fractions, we obtain that
$$
 \Big(x-y +x^{-1}-y^{-1}\Big)R_2\left(x,y;q\right)=G\left(x;q\right)-G\left(y;q\right).
$$
Pairing the $j$ and $t-j$ terms in the sum in the last equality of the definition \eqref{eqn:ftrsdef} of $f_t(r,s)$ and then using the symmetry of $R_2$ along with \eqref{eqn:geqn}, we obtain
\begin{align*}
f_t(1,2)&=\frac{1}{2t}\sum_{j=1}^{t-1} \left(\zeta_t^{-j}-\zeta_t^{-2j}+\zeta_t^{j}-\zeta_t^{2j}\right)R_2\left(\zeta_t^j,\zeta_t^{2j};q\right)= \frac{1}{2t}\left(\sum_{j=0}^{t-1} G\left(\zeta_t^{j};q\right) -\sum_{j=0}^{t-1} G\left(\zeta_t^{2j};q\right)\right)\\
&
=\frac{1}{2t}\left({\sum_{j=0}^{t-1}} R\left(\zeta_t^{j};q\right) - 2\sum_{j=0}^{\frac{t}{2}-1}R\left(\zeta_{\frac{t}{2}}^j;q\right)\right)= \frac{1}{2}\left(g_t(0) - g_{\frac{t}{2}}(0)\right)=-\frac{1}{2} g_{t}\left(\frac{t}{2}\right).
\end{align*}
\end{proof}
We are now ready to move on to the proof of Theorem \ref{thm:FullRankIdentity} (2).
\begin{proof}[Proof of Theorem \ref{thm:FullRankIdentity} (2)]

Given Proposition \ref{prop:FullRankEvenEquality}, it is clear that for $t$ even
\begin{equation}\label{eqn:N2evenrelate}
\NF_2(1,t;n)-\NF_2(2,t;n)=-\frac{1}{2}N\left(\frac{t}{2},t;n\right)\leq 0.
\end{equation}
with equality if and only if there are no partitions of $n$ with rank congruent to $\frac{t}{2}$ modulo $t$.  Whenever $n\leq \frac{t}{2}$ it is clear by \eqref{eqn:N2evenrelate} that 
$$
\NF_2(1,t;n)=\NF_2(2,t;n).
$$
Now assume that $n\geq \frac{t}{2}+1$ with $n\neq \frac{t}{2}+2$.  If $n=2m-1 + \frac{t}{2}$ for some $m\in \N$, then the partition $\left(m+\frac{t}{2}, 1,\dots,1\right)$ with precisely $m-1$ parts of size $1$, has rank equal to $\frac{t}{2}$.  If $n=2m+\frac{t}{2}$ with $2\leq m\in \N$, then the partition $\left(m+\frac{t}{2},2,1,\dots,1\right)$ with precisely $m-2$ parts of size $1$, has rank equal to $\frac{t}{2}$.  Since the only partition of $\frac{t}{2}+2$ with rank at least $\frac{t}{2}$ is $\left(\frac{t}{2}+2\right)$, we find that there are no partitions of size $\frac{t}{2}+2$ with rank equal to $\frac{t}{2}$ or $-\frac{t}{2}$.  Thus the set of $n$ with $\NF_2(1,t;n)=\NF_2(2,t;n)$ is precisely $\left\{ 0,\dots,\frac{t}{2},\frac{t}{2}+2 \right\}$. 

\end{proof}

\section{Atkin and Swinnerton-Dyer type Identities}\label{section:ASD-identities}
Using Proposition \ref{prop:GeneralSetupIdentity}, we are able to determine some infinite product and related identities by using the results of Atkin and Swinnerton-Dyer \cite{AtkinSwinnertonDyer1} for $t=5$ and $t=7$.

We first let $t=5$.  The identities which we will obtain in Theorem \ref{thm:IdentityTheorem5} (1) were proven in Theorem 17 of Andrews \cite{Andrews1}.  The remaining identities in Theorem \ref{thm:IdentityTheorem5} were proven (with a different method) by Keith \cite{Keith}.  However, we include this case for completeness as well as to exhibit this method of constructing identities.  Theorem \ref{thm:FullRankIdentity} (1) and the conjugation symmetry from \eqref{eqn:fullrankconjugation} immediately implies, as previously shown in Theorem 17 of Andrews \cite{Andrews1}, that 
\begin{equation}\label{eqn:5rsneq0}
f_5\left(r,s\right)=0.
\end{equation}
unless $r=0$ or $s=0$.  Thus the only remaining cases are $0=r<s$.  From Corollary \ref{cor:General47Identity} and the symmetry \eqref{eqn:conjugation}, we have
\begin{equation}\label{eqn:Identity5}
f_5(0,s) = g_5(1,2).
\end{equation}
Combining \eqref{eqn:Identity5} with Theorem 4 of Atkin and Swinnerton-Dyer implies the following result.
\begin{theorem}\label{thm:IdentityTheorem5}
The following equalities hold for $t=5$.
\begin{enumerate}
\item For every $1\leq s\leq 4$ we have
$$
f_{5,1}(0,s)=f_{5,4}(0,s)=0.
$$
\item For every $1\leq s\leq 4$ we have
$$
f_{5,2}(0,s) = \frac{\left(q^5;q^5\right)_{\infty}} { \left(q^2;q^5\right)_{\infty} \left(q^3;q^5\right)_{\infty}}.
$$
\item For every $1\leq s\leq 4$ we have
$$
f_{5,0}(0,s)
= \frac{q}{\left( q^5;q^5 \right)_{\infty}}
\sum_{n \in \Z} 
\frac{\left(-1\right)^n q^{  \frac{15n}{2} \left(n+1\right) }}{1-q^{5n+1}}.
$$
\end{enumerate}
\end{theorem}
\begin{remark}
For $1\leq s\leq 4$, Theorem \ref{thm:IdentityTheorem5} (2) implies that $f_{5,2}\left(0,s\right)$ is a weakly holomorphic modular form of weight $\frac{1}{2}$, while Theorem \ref{thm:IdentityTheorem5} (3) implies that $f_{5,0}\left(0,s\right)$ is a mock theta function.
\end{remark}

We next turn to the case $t=7$.  The identities in Theorem \ref{thm:IdentityTheorem7} (1) other than $f_{7,0}\left(1,3\right)=0$ were already proven in Theorem 18 of Andrews \cite{Andrews1}.  Noting that $f_7(1,2)=0$ by Theorem \ref{thm:FullRankIdentity} (1) and the relation coming from conjugation, the only interesting cases which remain are $f_7(0,1)$ and $f_7(1,3)=f_7(2,3)$.  
\begin{theorem}\label{thm:IdentityTheorem7}
For $t=7$, we have the following identities for $f_{7,d}\left(r,s\right)$.
\begin{enumerate}
\item  For every $0\leq r<s\leq 3$ we have 
$$
f_{7,1}\left(r,s\right)=f_{7,5}\left(r,s\right)=f_{7,0}\left(1,3\right)=0.
$$
\item For $d=2$, we have 
$$
f_{7,2}(1,3)= -\frac{q^2}{\left(q^7;q^7\right)_{\infty}}\sum_{n\in \Z}\frac{\left(-1\right)^n q^{\frac{21n}{2}\left(n+1\right)}}{1-q^{7n+3}}.
$$
\item For $d=3$, we have 
$$
f_{7,3}(1,3)=-f_{7,3}(0,1)= \frac{ \left(q^7; q^7\right)_{\infty}}{ \left(q^2;q^7\right)_{\infty} \left(q^5;q^7\right)_{\infty}}.
$$
\item For $d=4$, we have 
$$
f_{7,4}(0,1)=-f_{7,4}(1,3)= \frac{\left(q^7;q^7\right)_{\infty}}  {\left(q^3;q^7\right)_{\infty} \left(q^4;q^7\right)_{\infty}}.
$$
\end{enumerate}
\end{theorem}
\begin{proof}
Corollary \ref{cor:General47Identity} and the symmetries in \eqref{eqn:conjugation} immediately imply
\begin{equation}\label{eqn:Identity7}
f_7(0,1)=f_7(7,8)=g_7(6,2)=g_7(1,2),\qquad\ \  f_7(2,3)=f_7(5,4)=-g_7(3,2)=g_7(2,3).
\end{equation}
The result then follows by Theorem 5 of \cite{AtkinSwinnertonDyer1}, where $g_{7,d}\left(1,2\right)$ and $g_{7,d}\left(2,3\right)$ are explicitly given for the choices of $d$ in (1), (2), (3), and (4).
\end{proof}

\section{Inequalities for the full rank}\label{section:FullRankInequality}
In this section, we will prove Theorem \ref{thm:GeneralFullRankInequality}.  For two $q$-series $f$ and $g$, we will abuse notation to use the abbreviations $f\gg g$ and $g\ll f$ to mean that for $n$ sufficiently large the $n$-th Fourier coefficient of $f$ is strictly greater than the $n$-th Fourier coefficient of $g$ (this will cause no confusion, since we do not require analytic bounds of this type within this paper).  Using this notation the statement of Theorem \ref{thm:GeneralFullRankInequality} may be rewritten as $f_t(r,s)\gg 0$.
\begin{proof}[Proof of Theorem \ref{thm:GeneralFullRankInequality}]
It is sufficient to show that $f_t(\evar,\evar+1)\gg 0$ for $\evar=0$ and for every $1< \evar <\frac{t-1}{2}$.  Indeed, this follows since 
$$
f_t(r,s)=f_t(r,r+1)+f_t(r+1,r+2)+\dots + f_t(s-1,s)
$$
and $f_t(1,2)=0$ by Theorem \ref{thm:FullRankIdentity} (1).  We will prove $f_t(\evar,\evar+1)\gg 0$ separately for the congruence classes $\evar\equiv 0,1,2\pmod{3}$.  Since $(t,3)=1$, the congruence classes $\evar\equiv 1\pmod{3}$, $\evar\equiv t-2\pmod{3}$, and $\evar\equiv 1-t\pmod{3}$ are distinct and hence cover all possible congruence classes.  We will prove in each case that $f_t\left(d,d+1\right)\gg 0$ by making an appropriate choice of $r\equiv 1\pmod{3}$ satisfying the conditions of Lemma \ref{lem:1mod3} and then using the symmetries of $f_t$ to relate $f_t(\evar,\evar+1)$ and $f_t(r,r+1)$.  We begin with the case $\evar\equiv 1\pmod{3}$.  In this case, we choose $r=\evar$ and will find that $f_t(r,r+1)\gg 0$ will even hold in the slightly more general setting where we allow $3\mid t$.
\begin{proposition}\label{prop:GeneralOneMod3Inequality}
Suppose that $t>9$ is an odd integer and $1<r<\frac{t-1}{2}$ is an integer satisfying $r\equiv 1\pmod{3}$.  Then  
$$
f_t\left(r,r+1\right)\gg 0.
$$ 
\end{proposition}
\begin{proof}
We first use Lemma \ref{lem:1mod3} to rewrite $f_t\left(r,r+1\right)$ as 
\begin{equation}\label{eqn:ftrr1}
f_t\left(r,r+1\right) = \sum_{m=0}^{\left\lceil \frac{r-1}{6}\right\rceil -1} g_t\left(r-1 -3m,\frac{ t-3}{2} - 3m\right).
\end{equation}
For $0\leq m\leq \left\lceil\frac{r-1}{6}\right\rceil-1$, after checking that the conditions of Theorem \ref{thm:RankDifferences} are satisfied, one obtains that 
\begin{equation}\label{eqn:gtgg1mod3}
g_t\left(r-1-3m,\frac{t-3}{2}-3m\right)\gg 0.
\end{equation}
Combining \eqref{eqn:gtgg1mod3} and \eqref{eqn:ftrr1} implies that $f_t\left(r,r+1\right)\gg 0$, completing the proof.
\end{proof}
We next consider the case $\evar\equiv t-2\pmod{3}$. 
\begin{lemma}\label{lem:rt-2mod3}
Suppose that $t>9$ is odd with $(t,3)=1$ and $0\leq \evar<\frac{t-1}{2}$ satisfies $\evar\equiv t-2\pmod{3}$.  Then $f_t(\evar,\evar+1)\gg 0$. 
\end{lemma}
\begin{proof}
Setting $r:=t-1-\evar$, one sees immediately that $r\equiv 1\pmod{3}$ and $\frac{t-1}{2}< r \leq t-1$. By \eqref{eqn:fullrankconjugation}, it follows that
$$
f_t\left(r,r+1\right) = -f_t\left(\evar,\evar+1\right).
$$
Hence showing that $f_t(\evar,\evar+1)\gg 0$ is equivalent to showing that $f_t(r,r+1)\ll 0$. 

Following the notation from the proof of Lemma \ref{lem:1mod3}, we write $r=1+3j$.  We expand \eqref{eqn:rmod1arbitrary} and split into two sums to obtain
\begin{equation}\label{eqn:toexpand}
f_t(r,r+1) = \sum_{m=0}^{\left\lceil \frac{j}{2}\right\rceil -1} g_t(3j -3m) - \sum_{m=0}^{\left\lceil \frac{j}{2}\right\rceil -1} g_t\left(\frac{ t-3}{2}  - 3m\right).
\end{equation}
For $t\equiv \pm 1 \pmod{3}$, we define $\, \ell:=j-\frac{t\mp 1}{6}\in \Z$ and then make the change of variables $m\to m+\ell$ in the first sum of \eqref{eqn:toexpand} to obtain
\begin{equation}\label{eqn:toexpand2}
f_t(r,r+1) = \sum_{m=-{\ell}}^{\left\lceil \frac{j}{2}\right\rceil -1-\ell}g_t\left( \frac{t\mp 1}{2} -3m\right) - \sum_{m=0}^{\left\lceil \frac{j}{2}\right\rceil -1} g_t\left(\frac{ t-3}{2}  - 3m\right).
\end{equation}

First consider the case $t\equiv 1\pmod{3}$.  A straightforward calculation shows that the reverse inequality of the conditions of Theorem \ref{thm:RankDifferences} is satisfied for the difference of the terms $0\leq m\leq \left\lceil \frac{j}{2}\right\rceil -1-\ell$ coming from the two sums in \eqref{eqn:toexpand2}, so that 
\begin{equation}\label{eqn:posll}
\sum_{m=0}^{\left\lceil \frac{j}{2}\right\rceil -1-\ell} g_t\left(\frac{t-1}{2} -3m,\frac{t-3}{2}-3m\right)  \ll 0.
\end{equation}
Combining \eqref{eqn:posll} with \eqref{eqn:toexpand2} gives
\begin{equation}\label{eqn:bndt1}
f_t(r,r+1) \ll \sum_{m=-{\ell}}^{-1}g_t\left( \frac{t- 1}{2} -3m\right) - \sum_{m=\left\lceil \frac{j}{2}\right\rceil -\ell}^{\left\lceil \frac{j}{2}\right\rceil -1} g_t\left(\frac{ t-3}{2}  - 3m\right).
\end{equation}
In the first sum of \eqref{eqn:bndt1} we make use of the symmetry \eqref{eqn:conjugation} and then change variables $m\to -m$, while in the second sum of \eqref{eqn:bndt1} we make the change of variables $m\to m+\left\lceil \frac{j}{2}\right\rceil -1-\ell$.  Thus \eqref{eqn:bndt1} becomes
\begin{equation}\label{eqn:bndt12}
f_t\left(r,r+1\right) \ll \sum_{m=1}^{\ell}g_t\Bigg(\frac{t+1}{2} -3m, \frac{t-3}{2}-3\left(\left\lceil\frac{j}{2}\right\rceil-1-\ell\right) -3m\Bigg).
\end{equation}
After carefully checking the necessary boundary conditions, we use Theorem \ref{thm:RankDifferences} again for each $1\leq m\leq \ell$ to establish that $f_t\left(r,r+1\right)\ll 0$.

For the case $t\equiv -1\pmod{3}$, we follow a similar argument.  We begin by making the change of variables $m\to -m$ in the first sum of \eqref{eqn:toexpand2} and then use conjugation \eqref{eqn:conjugation} to rewrite \eqref{eqn:toexpand2} as 
\begin{equation}\label{eqn:toexpandt-1}
f_t(r,r+1) = \sum_{m=\ell - \left\lceil \frac{j}{2}\right\rceil +1}^{\ell} g_t\left( \frac{t- 1}{2} -3m\right) - \sum_{m=0}^{\left\lceil \frac{j}{2}\right\rceil -1} g_t\left(\frac{ t-3}{2}  - 3m\right).
\end{equation}
We then employ Theorem \ref{thm:RankDifferences} to establish
$$
\sum_{m=0}^{\ell} g_t\left(\frac{t- 1}{2} -3m,\frac{t-3}{2}  - 3m\right)\ll 0.
$$
Thus \eqref{eqn:toexpandt-1} can be bounded by 
\begin{equation}\label{eqn:negll}
\sum_{m=\ell - \left\lceil \frac{j}{2}\right\rceil +1}^{-1} g_t\left( \frac{t- 1}{2} -3m\right) - \sum_{m=\ell+1}^{\left\lceil \frac{j}{2}\right\rceil -1} g_t\left(\frac{ t-3}{2}  - 3m\right).
\end{equation}
We make the change of variables $m\to -m$ and use conjugation \eqref{eqn:conjugation} in the first sum of \eqref{eqn:negll} while shifting the second sum of \eqref{eqn:negll} by $m\to m + \ell$.  Thus \eqref{eqn:negll} can be rewritten as
\begin{equation}\label{eqn:negll2}
\sum_{m=1}^{\left\lceil \frac{j}{2}\right\rceil-\ell- 1}g_t\left( \frac{t+1}{2}-3m, \frac{t-3}{2} - 3\ell -3m\right).
\end{equation}
After carefully checking that the boundary conditions are satisfied, we use Theorem \ref{thm:RankDifferences} once more for each $1\leq m\leq \left\lceil \frac{j}{2}\right\rceil-\ell- 1$ to establish the result.
\end{proof}

We now move on to the final lemma which we will require to cover all possible choices of $d$ modulo 3 in the proof of Theorem \ref{thm:GeneralFullRankInequality}.
\begin{lemma}\label{lem:r1-tmod3}
Suppose that $t>9$ is odd with $(t,3)=1$ and $0\leq \evar<\frac{t-1}{2}$ satisfies $\evar\equiv 1-t\pmod{3}$.  Then $f_t(\evar,\evar+1)\gg 0$.
\end{lemma}
\begin{proof}
We set $r:=\evar+t$ so that $r\equiv 1\pmod{3}$, $t\leq r < \frac{3t-1}{2}$, and
$$
f_t\left(r,r+1\right)=f_t\left(\evar,\evar+1\right).
$$
Therefore, it suffices to show that $f_t(r,r+1)\gg 0$ with the given boundary conditions on $r$.  

We begin by using Lemma \ref{lem:1mod3} to write $f_t(r,r+1)$ as two sums and then shift both sums to be of type
\begin{equation}\label{eqn:hdef}
\sum_{m=-a}^{b} g_t\big(c-3m\big)
\end{equation}
with $c\geq 0$ as small as possible.  To this end, we choose $d_1$ and $d_2$ so that 
\begin{align}\label{eqn:c1def}
c_1&:=t-(r-1)+3d_1\in \{0,1,2\},\\
\label{eqn:c2def}
c_2&:=\frac{t-3}{2}-3d_2\in \{0,1,2\}.
\end{align}
Since $r-1\equiv 0\pmod{3}$, we see that $c_1\equiv t\pmod{3}$ and $c_2\equiv -t\pmod{3}$.  Using the fact that $(t,3)=1$, we see that $c_1+c_2=3$.  
Using conjugation \eqref{eqn:conjugation} and shifting $m\to -m+d_1$ in the first sum from Lemma \ref{lem:1mod3} and $m\to m+d_2$ in the second sum yields
\begin{equation}\label{eqn:Bh}
f_t\left(r,r+1\right)=\sum_{m=d_1-\left\lceil\frac{j}{2}\right\rceil+1}^{d_1}g_t\left(c_1-3m\right)  - \sum_{m=-d_2}^{\left\lceil\frac{j}{2}\right\rceil-d_2-1}g_t\left(c_2-3m\right),
\end{equation}
where we denote $r=1+3j$ as in the proof of Lemma \ref{lem:1mod3}.  
We now split the first sum into two sums with the terms $m\leq 1$ and $m>1$, while splitting the second sum into the two sums separated by $m< 0$ and $m\geq 0$.  
We next make the change of variables $m\to m+1$ in the sum of terms $m\leq 1$ from the first sum of \eqref{eqn:Bh} and also $m\to m+1$ in the sum of terms $m<0$ from the second sum of \eqref{eqn:Bh}.  Using $c_1+c_2=3$, we can rewrite $f_t\left(r,r+1\right)$ as 
\begin{equation}\label{eqn:aftersplit}
\sum_{m=d_1-{\left\lceil\frac{j}{2}\right\rceil}}^0 g_t\left( -c_2-3m\right) - \sum_{m=0}^{\left\lceil\frac{j}{2}\right\rceil-d_2-1} g_t\left(c_2-3m\right) + \sum_{m=2}^{d_1} g_t\left(c_1-3m\right) -\sum_{m=-d_2-1}^{-2} g_t\left(-c_1-3m\right).
\end{equation}
We now make the change of variables $m\to -m$ and use conjugation \eqref{eqn:conjugation} in the first and last sums of \eqref{eqn:aftersplit}, yielding
\begin{equation}\label{eqn:match}
\sum_{m=0}^{\left\lceil\frac{j}{2}\right\rceil-d_1} g_t\left(c_2-3m\right) - \sum_{m=0}^{\left\lceil\frac{j}{2}\right\rceil-d_2-1} g_t\left(c_2-3m\right) + \sum_{m=2}^{d_1} g_t\left(c_1-3m\right)- \sum_{m=2}^{d_2+1} g_t\left(c_1-3m\right).
\end{equation}
Comparing \eqref{eqn:c1def} with \eqref{eqn:c2def} along with the restriction $t\leq r < t+\frac{t-1}{2}$, we have $0\leq d_1< d_2$.  We cancel terms in the first two and last two summands of \eqref{eqn:match} to rewrite 
\begin{equation}\label{eqn:cancel}
f_t\left(r,r+1\right)=\sum_{m=\left\lceil\frac{j}{2}\right\rceil-d_2}^{\left\lceil\frac{j}{2}\right\rceil-d_1} g_t\left(c_2-3m\right) - \sum_{m=d_1+1}^{d_2+1}g_t\left(c_1-3m\right).
\end{equation}
Next we make the change of variables $m\to -m-\left\lceil\frac{j}{2}\right\rceil+d_2$ in the first sum of \eqref{eqn:cancel} and the change of variables $m\to -m-\left(d_1+1\right)$ in the second sum and then use conjugation \eqref{eqn:conjugation} on the terms in both sums.  Again noting the fact that $c_1+c_2=3$, this yields
\begin{equation}\label{eqn:tobound}
f_t\left(r,r+1\right)=\sum_{m=0}^{d_2-d_1} g_t\left( 3\left\lceil\frac{j}{2}\right\rceil-3-3d_2+c_1 +3m, 3d_1+c_2 +3m\right).
\end{equation}
We see by \eqref{eqn:c1def} and \eqref{eqn:c2def} that $d_1=\frac{r-1}{3}+\left\lfloor -\frac{t}{3}\right\rfloor$ and $d_2=\left\lfloor\frac{t-3}{6}\right\rfloor$, which may be used to easily verify the inequalities necessary to use Theorem \ref{thm:RankDifferences}.  Hence for $0\leq m\leq d_2-d_1$, we have that 
$$
g_t\left(3\left\lceil\frac{j}{2}\right\rceil-3-3d_2 +c_1 +3m, 3d_1+c_2 + 3m\right)\gg 0,
$$
from which \eqref{eqn:tobound} implies that $f_t(r,r+1)\gg 0$.  
\end{proof}

The proof of Theorem \ref{thm:GeneralFullRankInequality} now follows by Lemmas \ref{lem:1mod3}, \ref{lem:rt-2mod3}, and \ref{lem:r1-tmod3}.
\end{proof}

\section{Inequalities for small moduli}\label{section:smallinequalities}
For small choices of $t$ for which Theorem \ref{thm:GeneralFullRankInequality} does not apply, we get positivity or negativity results for the difference $\NF_2(r,t;n)-\NF_2(s,t;n)$ depending on the congruence class of $n$ modulo $t$.

We begin with the case $t=2$.  In this case, we only have to consider $(r,s)=(0,1)\equiv (2,1)\pmod{2}$.  Therefore we may use Theorem \ref{thm:FullRankIdentity} (2) to conclude that 
$$
\NF_2\left(0,2;n\right)\geq \NF_2\left(1,2;n\right),
$$
with strict inequality if and only if $n=2$ or $n>3$.

We next give the relevant inequalities when $t=4$.  In this case we only need to consider $0\leq r<s\leq 2$.  We have already shown in Theorem \ref{thm:FullRankIdentity} (2) that whenever $n> 4$ or $n=3$, the inequality
$$
\NF_2(1,4;n)<\NF_2(2,4;n)
$$
holds, while for $n\leq 2$ and $n=4$ we have equality.  We consider the remaining cases in the theorem below.
\begin{theorem} \label{thm:SmallInequality4}
\noindent

\noindent
\begin{enumerate}
\item
Suppose that $s\in\left\{1,2\right\}$.  Then for every $n\geq 1$ we have that
$$
\NF_2(0,4; 2n)>\NF_2(s,4;2n).
$$
\item
For $n\geq 4$ we have that
$$
\NF_2(0,4;2n+1)>\NF_2(1,4;2n+1).
$$
For $n=0$ and $n=2$ we have the equality, while for $n\in \{1,3\}$ we have the reverse inequality.
\item
For $n\geq 1$ one obtains that
$$
\NF_2(0,4;2n+1)<\NF_2(2,4;2n+1),
$$
while $\NF_2(0,4;1)=\NF_2(2,4;1)$.
\end{enumerate}
\end{theorem}
\begin{remark}
Since the differences in Theorem \ref{thm:SmallInequality4} (1) (resp. (3)) are always nonnegative (resp. nonpositive), it would be interesting to investigate whether the difference enumerates an interesting combinatorial statistic.
\end{remark}
\begin{proof}
We begin by using \eqref{eqn:ftrsdefmod3} to write
\begin{multline}\label{eqn:f4gen}
f_4(r,s) = \frac{1}{8} \Big( i^r+i^{-r} -i^s-i^{-s}\Big)R\left(i;q\right)
+\frac{1}{16} \bigg(-2  \Big( i^r+i^{-r} -i^s-i^{-s}\Big)+(-1)^s-(-1)^r \bigg) R\left(-1;q\right) \\
+\frac{1}{16} \Big((-1)^r-(-1)^s \Big) R\left(1;q\right).
\end{multline}
We then rewrite $R\left(1;q\right)$, $R\left(i;q\right)$, and $R\left(-1;q\right)$ as a linear combination of $g_4\left(d\right)$ with $0\leq d\leq 2$ (after using conjugation \eqref{eqn:conjugation} as necessary), namely
\begin{align*}
R\left(1;q\right)&=g_4(0)+2g_4(1)+g_4(2),\\
R\left(-1;q\right)&=g_4\left(0,1\right)+g_4\left(2,1\right),\\
R\left(i;q\right)&=g_4(0,2).
\end{align*}
For $(r,s)=(0,2)$, simplification of \eqref{eqn:f4gen} yields
$$
f_4(0,2)= g_4(1,2).
$$
We now recall that Theorem 4 of Andrews and Lewis \cite{AndrewsLewis1} states that the $n$-th Fourier coefficient of $g_4(1,2)$ is positive for $n\geq 2$ even and negative for $n\geq 3$ odd.  The positivity (resp. negativity) of these Fourier coefficients hence establishes the $(r,s)=(0,2)$ case of part (1) (resp. part (3)). 

We now evaluate $f_4(0,1)$ by the above method and then split $g_4\left(a,b\right)=g_4(a)-g_4(b)$.  Since $g_4(1)=\frac{1}{2}g_2(1)$, simplification yields
\begin{equation}
\label{eqn:f401-2}
f_4(0,1)=g_4(1)-\frac{1}{2}g_4(2)=\frac{1}{2}\left(g_2(1)-g_4(2)\right).
\end{equation}
Since $N\left(2,4;2n\right)\geq 0$, (4.11) in Theorem 4 of Andrews and Lewis \cite{AndrewsLewis1} implies that 
$$
N\left(1,4;2n\right)-\frac{1}{2}N\left(2,4;2n\right)\geq N(1,4;2n)-N(2,4;2n)> 0
$$
whenever $n\geq 1$.  This combined with \eqref{eqn:f401-2} completes the proof of part (1).

In the remaining case, we cannot directly use the results of Andrews and Lewis, since their equation (4.12) only gives 
\begin{equation}\label{eqn:4diff}
N(1,4;2n-1)-N(2,4;2n-1)< 0
\end{equation}
whenever $n\geq 2$, while conversely the coefficients of $g_4(1)$ are nonnegative.  Hence we must compare the difference \eqref{eqn:4diff} with the coefficients of $g_4(1)$.  To do so, we will prove the following refinement of \eqref{eqn:4diff}.  We shall show that for all $n\geq 5$, one has
\begin{equation}\label{eqn:AndrewsLewisRefine}
N\left(1,4;2n-1\right)<N\left(2,4,2n-1\right)<2N\left(1,4;2n-1\right).
\end{equation}
By \eqref{eqn:f401-2}, we have
$$
f_4(0,1)=\frac{1}{2}\sum_{n=1}^{\infty} \Big(N(1,2;n)-N(2,4;n)\Big)q^n.
$$
We next show that for $n\geq 13$ odd we have 
\begin{equation}\label{eqn:N124}
N\left(1,2;n\right)- N\left(2,4;n\right) > 0,
\end{equation}
and the remaining cases will then follow by directly computing the first 13 coefficients of $f_4(0,1)$.  In order to show the inequality given in \eqref{eqn:N124}, we construct an injection of the partitions of $n$ with rank congruent to $2$ modulo $4$ to those with rank congruent to $1$ modulo $2$ whenever $n\geq 13$ odd.  We denote the largest  summand of a partition $\lambda$ of $n$ with rank congruent to $2$ modulo $4$ by $\lambda_1$ and the second largest by $\lambda_2$ and assume that $\lambda$ has $\ell$ parts. 
\begin{enumerate}
\item
If $\lambda_1\geq \lambda_2+2$ with $\left(\lambda_1,\lambda_2\right)\neq \left(3,1\right)$, then we change $\lambda_1\to \lambda_1-2$  and add in the extra summand $2$ to $\lambda$.
\item If $\left(\lambda_1,\lambda_2\right)=\left(3,1\right)$, then we change $\lambda_1\to \lambda_1+1$ and $\lambda_2\to \lambda_2+3$, while removing $4$ parts of size $1$ from $\lambda$.
\item If $\lambda_1<\lambda_2+2$ and $\lambda_2 \neq 1$, then we change $\lambda_2 \to \lambda_2-1$ and add in an extra summand $1$ to $\lambda$.
\item If $\lambda_1<\lambda_2+2$ and $\lambda_2=1$, then we change $\lambda_1 \to \lambda_1+6$ and $\lambda_2\to \lambda_2+3$, while removing $9$ parts of size $1$ from $\lambda$.
\end{enumerate}
It is easily verified that the partitions in the image of this mapping all have odd rank.  
Each of the cases (1)--(4) is obviously itself an injection, so we only need to check that the images of each are pairwise disjoint.  In case (1) the resulting partition has rank congruent to $-1$ modulo $4$.  Case (2) also yields partitions with rank congruent to $-1$ modulo 4, but with no parts of size 2, and hence its image is disjoint from the image in case (1).  In case (3) the rank is congruent to $1$ modulo $4$, so the rank splits the case (3) from (1) and (2).  In case (4) the partitions in the image do not have any parts of size $2$ and their ranks are congruent to $-1$ modulo $4$, splitting case (4) from cases (1) and (3).  To separate the cases (2) and (4) we note that the largest part occurs twice in case (4) and only once in case (2).

In order to get a strict inequality in \eqref{eqn:N124}, we must also show that this injection is not onto.  The image of our injection is restricted to partitions containing a part of size 1 or 2, forcing the desired strict inequality since for $n\geq 8$ one of the partitions $\left(n-3,3\right)$ and $\left(n-4,4\right)$ has rank congruent to $1$ modulo $2$ but no parts of size $1$ or $2$.

\end{proof}
The cases $t=5$ and $t=7$ will follow directly from the identities \eqref{eqn:Identity5} and \eqref{eqn:Identity7} in Theorems \ref{thm:IdentityTheorem5} and \ref{thm:IdentityTheorem7} combined with the inequalities given in Theorem 1.1 of \cite{BringmannKane1} (and the more precise version given in Tables 1 and 2 of the Appendix in \cite{BringmannKane1}).  However, as these are the only cases with $(t,6)=1$ which are not contained in Theorem \ref{thm:GeneralFullRankInequality}, we include the conclusions for completeness.
\begin{theorem} \label{thm:SmallInequality57}
\noindent
\begin{enumerate}
\item For $t=5$ we have the following inequalities:
\begin{itemize}
\item[(a)]
For $0<s\leq 4$ we have
$$
\NF_2\left(0,5;5n+2\right)\geq \NF_2\left(s,5;5n+2\right),
$$
where strict inequality is satisfied for $n>5$.
\item[(b)]
For $n\in \N$ and $0<s\leq 4$, one has the inequality
$$
\NF_2\left(0,5;5n\right)>\NF_2\left(s,5;5n\right).
$$
\item[(c)]
For every $n\geq 0$ and $0<s\leq 4$, it holds that 
$$
\NF_2\left(0,5;5n+3\right)<\NF_2\left(s,5;5n+3\right).
$$
\end{itemize}
\item
For $t=7$ the following inequalities hold:
\begin{itemize}
\item[(a)]  For $n\geq 20$, one obtains 
$$
\NF_2\left(0,7;7n\right)> \NF_2\left(1,7;7n\right)=\NF_2\left(3,7;7n\right).
$$
For all $n\geq 0$ we have
$$
\NF_2\left(0,7;7n+2\right)>\NF_2\left(1,7;7n+2\right).
$$
When $n>7$ we have 
$$
\NF_2\left(0,7;7n+3\right)<\NF_2\left(1,7;7n+3\right),
$$
$$
\NF_2\left(0,7;7n+4\right)>\NF_2\left(1,7;7n+4\right).
$$
For $n>4$ one has
$$
\NF_2\left(0,7;7n+6\right)<\NF_2\left(1,7;7n+6\right).
$$
\item[(b)]
If $(r,s)=(0,3)$ then for every $n\geq 0$ one has 
$$
\NF_2\left(0,7;7n+2\right)>\NF_2\left(3,7;7n+2\right),
$$
and for $n>1$  
$$
\NF_2\left(0,7;7n+6\right)<\NF_2\left(3;7n+6\right).
$$
\item[(c)]
For $(r,s)=(1,3)$ and $n>7$, we have 
$$
\NF_2\left(1,7;7n+2\right)<\NF_2\left(3,7;7n+2\right),
$$
$$
\NF_2\left(1,7;7n+3\right)>\NF_2\left(3,7;7n+3\right),
$$
$$
\NF_2\left(1,7;7n+4\right)<\NF_2\left(3,7;7n+4\right).
$$
Finally for all $n\geq 0$ we have 
$$
\NF_2\left(1,7;7n+6\right)>\NF_2\left(3,7;7n+6\right).
$$
\end{itemize}
\end{enumerate}
\end{theorem}

\end{document}